\documentclass[12pt,a4paper]{article}
\usepackage{amsmath,amssymb,amsthm,graphicx,color}
\usepackage[left=1in,right=1in,top=1in,bottom=1in]{geometry}
\usepackage{cite}
\usepackage[british]{babel}

\newtheorem{theo}{Theorem}[section]
\newtheorem{lm}{Lemma}[section]

\allowdisplaybreaks

\numberwithin{equation}{section}

\newcommand{\R}{{\mathbb R}}
\newcommand{\Z}{{\mathbb Z}}
\newcommand{\ZP}{{\Z_+}}
\newcommand{\N}{{\mathbb N}}

\newcommand{\Exp}{{E}}
\newcommand{\Expo}{{\Exp_\omega}}
\renewcommand{\Pr}{{P}}
\newcommand{\Pro}{{\Pr_\omega}}
\newcommand{\bbP}{{\mathbb{P}}}
\newcommand{\bbE}{{\mathbb{E}}}

\newcommand{\1}{{\mathbf 1}}
\newcommand{\Var}{\mathop{\mathbb{V}{\rm ar}}\nolimits}
\newcommand{\eps}{\varepsilon}

\newcommand{\F}{{\mathcal F}}

\title{Random walk in mixed random environment without uniform ellipticity}
\author{Ostap Hryniv \and Mikhail V.\ Menshikov \and Andrew R.\ Wade \and {}
\normalsize{Durham University, South Road, Durham DH1 3LE, UK.}}
\begin{document}

\maketitle 

\begin{abstract}
We study a random walk in random environment on $\ZP$. The random environment is not homogeneous in law, but is
a mixture of two kinds of site, one in asymptotically vanishing proportion. The two kinds of site are (i) points endowed with probabilities drawn from a
symmetric
distribution with heavy tails at $0$ and $1$, and (ii) `fast points' with a fixed systematic drift. Without these fast points, the model is related to 
the diffusion in heavy-tailed (`stable') random potential studied by Schumacher and Singh; the fast points perturb that model.
 The two
components compete to determine the behaviour of the random walk;
we identify phase transitions in terms of
 the model parameters.
We give conditions for recurrence and transience, and prove almost-sure
bounds for the trajectories of the walk. 
\end{abstract}

\vskip 2mm


\vskip 2mm


\section{Introduction}
\label{int}

Given an infinite sequence $\omega = (p_0,p_1,p_2,\ldots)$ such
that  $p_i \in (0,1)$ for
all  $i \in \ZP :=
\{0,1,2,\ldots\}$, we consider $X := (X_t)_{t \in
\ZP}$ 
a nearest-neighbour random walk on $\ZP$ 
defined as follows. Set $X_0 =0$, and for $n\in \N := \{1,2,\ldots\}$, 
\begin{align} \label{1006bb}
 \Pro [ X_{t+1} = n-1  \mid X_t  = n] & =   p_n,\nonumber\\
 \Pro [ X_{t+1}  = n+1  \mid X_t  = n] & =   1-p_n =:q_n,
\end{align}
 and $\Pro [ X_{t+1} =0 \mid X_t =0] = p_0$, $\Pro [ X_{t+1} =1  \mid
X_t =0] = 1-p_0 =: q_0$.
The assumption that each $p_i$ be in $(0,1)$,  with the
 given
 form for the reflection at the origin, 
ensures that $X$ is an irreducible, aperiodic Markov chain on $\ZP$ under the {\em quenched} law $\Pro$.

The sequence of jump probabilities $\omega$ is the {\em environment} for the random walk. 
We take $\omega$ itself to be random --- then
$X$ is a {\em random walk in random environment}
(RWRE). Specifically, $p_0,p_1,\ldots$ will be
 a sequence of independent (not necessarily i.i.d.)\
 $(0,1)$-valued random variables 
on a probability space $(\Omega, \F,\bbP)$. 

There has been much
 interest in the RWRE recently; see for
example \cite{rev} or \cite{zeit} for surveys. Many
of the results
in the literature assume that the $p_i$ are i.i.d.\ and that
their tails   are `light' at $0$ and at $1$.
In particular, much of the  
work on RWRE imposes a {\em uniform ellipticity} condition on the environment, that is
$\bbP [ p_0 \in (\eps, 1-\eps) ] =1$ for some $\eps>0$. 

In the case of an i.i.d.\ random environment, 
classical work of Solomon \cite{solomon} demonstrated the importance
 of the 
quantity $\bbE [ \log (p_0/q_0)]$, where $\bbE$ denotes expectation under $\bbP$; when that expectation exists Solomon \cite[Theorem 1.7]{solomon}
showed that its sign determines the recurrence classification of $X$.
The null-recurrent case in which $\bbE [ \log (p_0/q_0)] =0$ is known as
{\em Sinai's regime} after  Sinai's remarkable result \cite{sinai}
that $(\log t)^{-2} X_t$ converges weakly to some non-degenerate limit under the {\em annealed} probability measure.
Sharp results on the almost-sure behaviour of the RWRE in Sinai's regime   were provided
by Hu and Shi \cite{hushi},
making use of some delicate technical estimates involving the potential associated with the
random environment.
In all these results, $\bbE [ ( \log (p_0/q_0) )^2] < \infty$. 

Again in the i.i.d.\ setting, different behaviour is observed when $\log (p_0/q_0)$ is {\em heavy tailed},
in particular, when $\bbE [ (\log (p_0/q_0))^2] = \infty$,
i.e., the tails of $p_0$ are (very) heavy
approaching $0$ or $1$. The natural model in this heavy-tailed setting takes $\log (p_0/q_0)$
to be in the domain of attraction of a stable law. Singh \cite{singh1} gave analogues of the almost-sure results of Hu and Shi \cite{hushi}
in the stable law setting. Analogues of Sinai's weak convergence result
were obtained by Schumacher \cite{schumacher,schumacher0} and Kawazu {\em et al.} \cite{ktt}: now $(\log t)^{-\alpha} X_t$ has an
annealed
weak limit, where $\alpha \in (0,2]$ is the index of the stable law. 
 
Our main interest here is when
 the random environment
is {\em not} i.i.d., that is, {\em non-homogeneous}. In particular, we are interested in how the 
phenomena just described, namely, the recurrence classification of
Solomon \cite{solomon} and the almost-sure results of Hu and Shi \cite{hushi} and Singh \cite{singh1},
 are affected
by 
perturbations to the random environment (in a sense we describe in more detail below).
Perturbations of random walks in Sinai-type random environments were considered in \cite{mw1,mw2} and  \cite{gps}.
In contrast,  our main interest is perturbing the heavy-tailed
setting of Singh \cite{singh1}. We identify  families of perturbations 
whose natural parameters 
exhibit phase transitions in the asymptotic behaviour of the random walk.
In this
 sense, we identify the robustness of certain features of the model under perturbations.

The results of \cite{hushi,singh1,ktt} (as well as the annealed results of \cite{schumacher,schumacher0})
all either are stated in the continuous setting of diffusions in random potentials, or else
have their main proofs based in the continuous setting and then some extra approximation work to deduce results
in the discrete RWRE setting. These arguments, often making use of potential-theoretic methods, 
can be long and technical. In the present
 paper we take a different approach.
We study discrete processes directly via discrete ideas, as seems natural,
 rather than appealing to a diffusion approximation.

Our methods are not as sharp as those say of \cite{hushi,singh1}, 
but they are simpler and more robust in our non-homogeneous
random environment setting.
The methods of the present paper are based on Lyapunov function ideas;
the usefulness of Lyapunov function techniques for RWRE has already been demonstrated
in \cite{cmp,mw1,mw2}. In the present paper we develop these techniques further. In particular,
we give methods for proving almost-sure bounds for RWRE that improve on those in \cite{cmp,mw2}.

In the next section, Section \ref{sec:results}, we describe our model and present our main results.
We give the proofs of the results in Section \ref{sec:proofs}, after collecting some preliminaries, including
our main technical tools, in Section \ref{sec:prelim}.

\section{Model, results, and discussion}
\label{sec:results}

\subsection{Model description}
\label{sec:model}

We consider the RWRE as described in (\ref{1006bb}), so that  $p_0, p_1, \ldots$ are independent 
$(0,1)$-valued random variables under
$\bbP$.
Our random environment will in general be non-homogeneous, so that the law of $p_n$ will depend on $n$.

Fix $\delta \in (-1,1)$ and let $\phi : \ZP \to [0,1]$. Let $\chi_0,\chi_1,\ldots$ be independent $\{0,1\}$-valued random 
variables with $\bbP [ \chi_n = 1] = \phi(n)$,
 and let $\xi_0,\xi_1,\ldots$ be i.i.d.\ $(0,1)$-valued random variables, independent
of the $\chi_i$. We take
\[ p_n = \left( \frac{1-\delta}{2} \right) \chi_n + \xi_n (1-\chi_n) ,\]
so that
\[ q_n = \left( \frac{1+\delta}{2} \right) \chi_n + (1 - \xi_n)  (1-\chi_n) ,\]
and, since $\chi_n$ is  $\{0,1\}$-valued,
 $\log (p_n/q_n) = \rho \chi_n + \zeta_n (1 - \chi_n )$, where
\[ \rho := \log \left( \frac{1-\delta}{1+\delta} \right) \in \R; ~~~
\zeta_n := \log \left( \frac{\xi_n}{1-\xi_n} \right) .\]
Note that $\zeta_0, \zeta_1, \ldots$ is an i.i.d.\ sequence under $\bbP$.

If $\phi(n) \equiv 0$, this model is the classical random walk in i.i.d.\ random environment.
The other extreme $\phi(n) \equiv 1$ gives a reflecting random walk in a non-random environment
 with a fixed drift $\delta$ (when not at $0$). 
We consider perturbations of both of these extremes, with $\phi (n) \to 0$ or
$\phi(n) \to 1$, allowing us to interpolate between regimes of behaviour.
The case $\phi (n) \to 0$ represents a disordered medium 
that is `doped' with the  introduction of an asymptotically vanishing proportion of `fast points' with a fixed drift $\delta$.
The case $\phi(n) \to 1$ corresponds to a largely homogeneous system with a vanishing proportion of random
impurities (which, due to the heavy tails, can be `diode-like').

Often, we make the following assumption.
\begin{itemize}
\item[(S)] The distribution
of $\zeta_0$ is {\em symmetric} about $0$, i.e., $\bbP [ \zeta_0 \leq - x] = \bbP [ \zeta_0 \geq x ]$, $x \geq 0$.
\end{itemize}
Note that in the case $\phi(n) \equiv 0$, 
(S) is equivalent to the assumption that $p_0$ has the same distribution
as $q_0$. The following result is essentially due to Solomon \cite{solomon}.

\begin{theo}
\label{thm2}
Suppose that (S) holds and  $\phi(n) \equiv 0$.
Then  $X$ is recurrent for $\bbP$-a.e.\ $\omega$.
\end{theo}

We investigate the effect of perturbing this situation with some $\phi(n) >0$.
For many of our results we assume that $\zeta_0$ has heavy tails; in particular that $\bbE [ \zeta_0^2 ] = \infty$.
In such cases, we will typically assume that, for some $c \in (0,\infty)$
and $\alpha \in (0,2)$, as $r \to \infty$,
\begin{equation}
\label{tail1} \bbP [ \zeta_0 > r ] \sim c r^{-\alpha} ;\end{equation}
here and elsewhere we use `$\sim$' in the standard sense, indicating that the ratio of the two sides tends to $1$.
Our methods could be extended to the more general case
$\bbP [ \zeta_0 > r ] =   r^{-\alpha} L(r)$,
for some slowly-varying function $L$, at the expense of complicating the statements and
introducing additional technicalities; for these reasons we restrict the presentation to the case (\ref{tail1}).

We present our main results in the next few subsections. Very different phenomena occur depending on whether
$\bbE [ | \zeta_0 |]$ is finite or not, so we separate these two cases.

\subsection{Recurrence classification when $\bbE[ | \zeta_0 | ] < \infty$}

If $\bbE[ | \zeta_0 | ] < \infty$, then we need $\phi(n) \to 0$ to obtain 
a range of phenomena. The next two results identify the scale of $\phi$ that leads to phase transitions in the model.
 For clarity of presentation, 
we separate the cases where $\bbE[   \zeta_0 ^2 ]$
does and does not exist.

\begin{theo}
\label{pert0}
Suppose that (S) holds, and that $\bbE [   \zeta_0 ^2] \in (0,\infty)$.
Suppose that $\phi(n) = n^{-\beta}$ for $\beta \in (0,1)$. 
\begin{itemize}
\item[(i)] Suppose that $\beta < \frac{1}{2}$. 
If $\delta<0$, then $X$ is recurrent for $\bbP$-a.e.\ $\omega$. 
If $\delta>0$, then $X$ is transient for $\bbP$-a.e.\ $\omega$.
\item[(ii)] Suppose that $\beta > \frac{1}{2}$. 
Then (for any $\delta$) $X$ is recurrent for $\bbP$-a.e.\ $\omega$.
\end{itemize}
\end{theo}

When $\bbE [ \zeta_0^2] < \infty$, the model is  a particular perturbation of Sinai's regime,
in which a vanishing fraction of the sites are replaced by `fast points' with drift $\delta$.
A different, but related, perturbation of Sinai's regime was introduced in \cite{mw1,mw2}, in which
the perturbation occurs at every site, but is asymptotically small in magnitude. In that case,
a perturbation of density about $n^{-1/2}$ is also critical: see Theorem 6 of \cite{mw1}.

When  $\bbE[ \zeta_0^2 ] =\infty$, the `critical exponent' is different, as shown by the next theorem.

\begin{theo}
\label{pert1}
Suppose that (S) holds, and that (\ref{tail1}) holds for $\alpha \in (1,2)$.
Suppose that $\phi(n) = n^{-\beta}$ for $\beta \in (0,1)$. 
\begin{itemize}
\item[(i)] Suppose that $\beta < 1-\frac{1}{\alpha}$. 
If $\delta<0$, then $X$ is recurrent for $\bbP$-a.e.\ $\omega$. 
If $\delta>0$, then $X$ is transient for $\bbP$-a.e.\ $\omega$.
\item[(ii)] Suppose that $\beta > 1-\frac{1}{\alpha}$. 
Then (for any $\delta$) $X$ is recurrent for $\bbP$-a.e.\ $\omega$.
\end{itemize}
\end{theo}

\subsection{Almost-sure bounds when $\bbE[ | \zeta_0 | ] < \infty$}

Next we state almost-sure bounds for the trajectory of the walk
in the recurrent cases identified by Theorems \ref{pert0} and \ref{pert1}.
Our results are of `$\limsup$'-type, i.e., for functions $b_+$ and $b_-$, with
$b_+(t) \to \infty$ and $b_-(t) \to \infty$, such that $b_-(t) \leq b_+ (t)$,
we show that $X_t \leq b_+ (t)$ all but finitely
often, but $X_t \geq b_- (t)$ infinitely often. In each case $b_+(t)$ and $b_-(t)$
are $(\log t)^{\theta +o(1)}$, where $\theta \in (1,2]$ depends on the parameters of the model.

\begin{theo}
\label{pert2a}
Suppose that (S) holds, and that   $\bbE [   \zeta_0^2] \in (0,\infty)$.
Suppose that $\phi(n) = n^{-\beta}$ for $\beta \in (0,1)$. 
\begin{itemize}
\item[(i)] Suppose that $\beta < \frac{1}{2}$ and $\delta<0$.  Then  
 for $\bbP$-a.e.\ $\omega$, $\Pro$-a.s.,
\[ \limsup_{t \to \infty} \frac{X_t}{( \log t)^{\frac{1}{1-\beta}} } = \left( \frac{1-\beta}{\rho} \right)^{\frac{1}{1-\beta}}.\]
\item[(ii)] Suppose that $\beta > \frac{1}{2}$. Then for any $\eps >0$,
for $\bbP$-a.e.\ $\omega$, $\Pro$-a.s., for all but finitely many $t \in \ZP$,
\[ X_t \leq (\log t)^{2} (\log \log t)^{2+\eps} .\]
On the other hand, for some absolute constant $c$, for $\bbP$-a.e.\ $\omega$, $\Pro$-a.s., for infinitely many $t \in \ZP$,
\[ X_t \geq c (\log t)^{2} \log \log \log t .\]
\end{itemize}
\end{theo}

Theorem \ref{pert2a} shows similar behaviour to the analogous results (Theorems 2 and 5) of \cite{mw2};
roughly speaking, the upper envelope of the walk lives on scale $(\log t)^{\theta + o(1)}$
where $\theta$ is the minimum of $2$ and $\frac{1}{1-\beta}$.
The next result shows that, again, the heavy-tailed case shows different behaviour: 
now the scale exponent $\theta$ is the minimum of $\alpha$ and $\frac{1}{1-\beta}$.

\begin{theo}
\label{pert2b}
Suppose that (S) holds, and that (\ref{tail1}) holds for $\alpha \in (1,2)$.
Suppose that $\phi(n) = n^{-\beta}$ for $\beta \in (0,1)$. 
\begin{itemize}
\item[(i)] Suppose that $\beta < 1-\frac{1}{\alpha}$ and $\delta<0$. Then  
 for $\bbP$-a.e.\ $\omega$, $\Pro$-a.s.,
\[ \limsup_{t \to \infty} \frac{X_t}{( \log t)^{\frac{1}{1-\beta}} } = \left( \frac{1-\beta}{\rho} \right)^{\frac{1}{1-\beta}}.\]
\item[(ii)] Suppose that $\beta > 1-\frac{1}{\alpha}$. Then for any $\eps >0$,
for $\bbP$-a.e.\ $\omega$, $\Pro$-a.s., for all but finitely many $t \in \ZP$,
\[ X_t \leq (\log t)^{\alpha} (\log \log t)^{2+\eps} .\]
On the other hand, for some absolute constant $c$, for $\bbP$-a.e.\ $\omega$, $\Pro$-a.s., for infinitely many $t \in \ZP$,
\[ X_t \geq c (\log t)^{\alpha} \log \log \log t .\]
\end{itemize}
\end{theo}

Part (ii) of Theorem \ref{pert2b} should be compared to Singh's result \cite[Theorem 1]{singh1},
which covers the case $\phi (n) \equiv 0$ (actually in the analogous diffusion setting) and gives
a sharp `$\limsup$' result with normalization $(\log t)^\alpha \log \log \log t$: see the discussion around
(\ref{singheq}) below. Theorem \ref{pert2b}(ii) shows that $(\log t)^\alpha$ remains the correct (coarse)
scale
when our perturbation is small enough. Singh's result suggests that it is the lower bound in Theorem \ref{pert2b} that is sharp.

\subsection{Results when $\bbE [ | \zeta_0 | ] =\infty$}

In the case where $\bbE [ | \zeta_0 | ] =\infty$, the influence 
of the heavy-tailed sites in the environment overwhelms the `fast points'
unless the   heavy-tailed sites are scarce. In this case, 
we need $\phi(n) \to 1$ to observe interesting results and to obtain a phase transition.
Now the environment consists of sites with fixed drift perturbed by an asymptotically
small fraction of very `heavy-tailed' sites.

We state just one result, the recurrence classification, in this regime. 
Almost-sure bounds could also be obtained by our methods.

\begin{theo}
\label{thm4}
Suppose that (S) holds, and that (\ref{tail1}) holds for $\alpha \in (0,1)$.
Suppose that $\phi(n) = 1- n^{-\beta}$ for $\beta \in (0,1)$.
\begin{itemize}
\item[(i)] Suppose that $\beta < 1-\alpha$. Then (for any $\delta$) $X$ is recurrent for $\bbP$-a.e.\ $\omega$. 
\item[(ii)] Suppose that $\beta > 1 - \alpha$.  
If $\delta <0$, then $X$ is recurrent for $\bbP$-a.e.\ $\omega$.
If $\delta >0$, then $X$ is transient for $\bbP$-a.e.\ $\omega$.
\end{itemize}
\end{theo}

Finally, we give one result in the {\em non-symmetric} case where (S) does not hold. 
For simplicity we give a result in the case of no perturbation ($\phi(n) \equiv 0$),
but when
$\bbE[ | \zeta_0 | ] = \infty$ and $\zeta_0$ has regularly-varying tails
of different orders in each direction.

\begin{theo}
\label{thm3}
Suppose that $\phi(n) \equiv 0$ and for $0< \alpha_+ < 1 \wedge \alpha_-$,  for $x >0$,
\begin{equation}
\label{tran3} \bbP [ \zeta_0 > x ] = x^{-\alpha_+} L_+(x), ~~~ \bbP [ \zeta_0 < - x] = x^{-\alpha_-} L_- (x) ,\end{equation}
where $L_+$ and $L_-$ are slowly-varying.
Then for $\bbP$-a.e.\ $\omega$, $X$ is transient.
\end{theo}

\subsection{Further remarks and open problems}

The case $\phi(n) \equiv 0$ reduces to the classical i.i.d.\ random environment
setting as studied by Solomon \cite{solomon} and others.

Here by far the most studied case has $\bbE [ \zeta_0^2 ] < \infty$.
Under this condition and (S), of course $\bbE [\zeta_0]=0$;
the case of an i.i.d.\ environment  in which $\bbE[ \zeta_0] = 0$ and $\bbE[ \zeta_0^2] =\sigma^2 \in (0,\infty)$
is Sinai's regime. Our model can be seen as a generalization of this setting.
 In
Sinai's regime for the RWRE on $\ZP$, sharp almost-sure upper and lower
bounds were given by Hu and Shi \cite{hushi}. 
In particular \cite[Theorem 1.3]{hushi} says that for $\bbP$-a.e.\ $\omega$, $\Pro$-a.s.,
\begin{equation}
\label{hsbound}
\limsup_{t \to \infty} \frac{ X_t }{(\log t)^2 \log \log \log t} = \frac{8}{\pi^2 \sigma^2} .
\end{equation}

It is worth comparing our results specialized to Sinai's regime with the result (\ref{hsbound}).
Our Lemma \ref{lemlow} below, together with the upper bound in Lemma \ref{Tbounds} and the
 Chung--Jain--Pruitt `other' law of the iterated logarithm \cite{chung,jp}, yields
for any $\eps >0$, for $\bbP$-a.e.\ $\omega$, $\Pro$-a.s.,
for infinitely many $t \in \ZP$,
\[ X_t \geq  (1-\eps) \frac{2}{\pi^2 \sigma^2} (\log t)^2 \log \log \log t ,\]
which is  a factor of $4$ away from Hu and Shi's \cite{hushi} sharp lower bound in (\ref{hsbound}).

The case where $\phi(n) \equiv 0$ but $\bbE [ \zeta_0^2 ] =\infty$ has been studied by Singh \cite{singh1},
who  gives a sharp version of our Theorem \ref{pert2b}(ii) in that case.
Singh's result \cite[Theorem 1]{singh1} is stated   for  a diffusion in a random potential,
but it is indicated \cite[p.\ 138]{singh1} that the result can be   adapted to the RWRE setting. Specifically,
under the conditions of the $\phi(n) \equiv 0$ case of Theorem \ref{pert2b}, 
Singh \cite[Theorem 1]{singh1} states that for $\bbP$-a.e.\ $\omega$, $\Pro$-a.s.,
\begin{equation}
\label{singheq}
 \limsup_{t \to \infty} \frac{X_t}{ (\log t)^{\alpha} \log \log \log t } = c_0,\end{equation}
for some absolute constant $c_0 \in (0,\infty)$ depending only on the law of the environment.

A natural question not addressed in our stated results concerns almost-sure bounds for the transient cases identified in Theorems \ref{pert0} and \ref{pert1}. Now, since $X_t \to \infty$, one wants
 almost-sure bounds of the form $b_-(t) \leq X_t \leq b_+ (t)$
all but finitely often. 
The transient cases introduce extra technicalities. One is that
an extra argument (cf the proof of Theorem 3 in \cite{mw2}) is needed for the `all but finitely often'
lower bound. The main complication, however, is due to the fact that the simple bounds on
the expected hitting times
$T(n)$ that we give in Lemma \ref{Tbounds} below are not sharp enough, and a more careful
analysis of the shifting sums in $T(n)$ is needed (see Section \ref{hitting} for definition of $T(n)$). 
      In \cite{mw2} the analogous analysis was
accomplished using some technical estimates (e.g.\ Lemmas 5 and 6 of \cite{mw2});
suitable versions of these results  
in the heavy-tailed setting of the present paper seem harder to obtain.
Rough calculations suggest that one should have results of the form
$X_t = (\log t)^{\theta +o(1)}$, where now
$\theta \in (1,\infty)$ is given by:
\begin{itemize}
\item $\theta = \frac{1}{\beta}$ in the transient case identified in Theorem \ref{pert0}(i);
\item $\theta = \frac{\alpha-1}{\beta}$ in the transient case identified in Theorem \ref{pert1}(i).
\end{itemize}
We leave these as open problems.

We briefly mention other related work in the literature.
 In \cite{singh2}, Singh gives a `slow transient' result
for a two-part potential similar to our model but in the diffusion case: the rough analogue of his model in our setting
has $\alpha \in (1,2)$, $\phi(n) \equiv \phi \in (0,1)$, and $\delta >0$.
Other recent work 
concerns RWRE without uniform ellipticity
in  {\em high} dimensions \cite{gz}  and  the related `random conductance' model
\cite{bouk};  other random media models for which heavy-tailed cases have been explored include the
`random traps' model \cite{fin}. 

\section{Preliminaries}
\label{sec:prelim}
 
\subsection{Lyapunov functions}
\label{hitting}

For fixed $\omega$, there are two very useful natural 
Lyapunov functions provided by the
fact that $X$ is a nearest-neighbour random walk under $\Pr_\omega$.
For fixed $\omega$, define for $i \in \ZP$,
\begin{align}
\label{Ddef}
D_i & := D_i (\omega) := \prod_{j=0}^{i-1} \frac{p_j}{q_j} ;\\
\label{Deldef}
 \Delta_i & := \Delta_i (\omega) :=  \sum_{j=0}^i q_{i-j}^{-1} \prod_{k=i-j+1}^{i}
 \frac{p_k}{q_k} ;
 \end{align}
 here and throughout
 we adopt the convention that an empty product is $1$ and an empty sum is $0$,
 so that $D_0 =1$ and $\Delta_0 = 1/q_0$.
  Now for $n \in \ZP$ define 
 \[ f(n) := f (n ; \omega ) := \sum_{i=0}^n D_i(\omega). \]
 Write $\Expo$ for expectation under $\Pro$.
 The function $f$ has a classical interpretation in terms of hitting
 probabilities. For our purposes, its usefulness stems from the following fact.
We write $\1\{ \, \cdot \,\}$ for the indicator function of the given event.
 
 \begin{lm}
 \label{flem}
 For fixed $\omega$, any $t \in \ZP$,
  and any $n \in \ZP$,
 \[ \Expo [ f (X_{t+1} ) - f( X_t ) \mid X_t = n]
 = q_0 \1 \{ n = 0 \} .\]
 \end{lm}
 \begin{proof}
 For $n \geq 1$, 
  \[ \Expo [ f (X_{t+1} ) - f( X_t ) \mid X_t = n]
 =  p_n f(n-1) + q_n f (n+1) - f(n) = q_n D_{n+1} - p_n D_n = 0 .\]
 On the other hand, $\Expo [ f (X_{t+1} ) - f( X_t ) \mid X_t = 0]
 = q_0 (f(1) - f(0) ) = q_0 D_0 = q_0$.
  \end{proof}
  
For $n \in \ZP$,
 denote the first hitting time of $n$ by
 \begin{equation}
 \label{taudef}
 \tau_n := \min \{ t \in \ZP : X_t = n \} . \end{equation}
Note that 
 $X_{\tau_n} = n$ a.s.\ and
 $\tau_n \neq \tau_m$ for any $n \neq m$.
  For $n \in \ZP$,  set 
\[ T(n) := T(n ; \omega ) := \Expo [ \tau_n ], \]
with $\tau_n$ as defined at (\ref{taudef}); 
 $T(n)$ is
the expected first hitting time of
 $n$ under $\Pro$ for fixed environment $\omega$.
The following result is classical. Recall that we assume $X_0 =0$.

\begin{lm}
\label{lemexp}  
For $n \in \ZP$, 
$T(n;\omega) = \sum_{i=0}^{n-1} \Delta_i(\omega)$, where 
 $\Delta_i$
is given  by (\ref{Deldef}).
\end{lm}
 
 A key property of $T(X_t ; \omega)$ 
is the
 following strict submartingale result.

\begin{lm}
\label{1002a}
For fixed $\omega$, any $t \in \ZP$ and any $n\in
\ZP$,
\begin{align} \label{1002b} \Expo [ T( X_{t+1} ) -
T(X_t) \mid X_t  = n ] = 1.\end{align}
\end{lm}
\begin{proof}
For $n \geq 1$, we have \begin{align*}   
\Expo [ T( X_{t+1} )
- T(X_t ) \mid X_t  = n ]  
& =   p_n (T(n-1)-T(n))
+q_n(T(n+1)-T(n)) \\
& =  q_n \Delta_n - p_n \Delta_{n-1} = 1,\end{align*} by (\ref{Deldef}).
Also,
\[ \Expo [ T( X_{t+1} )
- T(X_t) \mid  X_t  = 0 ] = q_0 T(1) = 1,\] since
$T(1)=\Delta_0=1/q_0$. \end{proof}

  \subsection{Quenched bounds on Lyapunov functions}
  
  In this section we give  elementary bounds on $f$ and $T$.
We start with $f$.
    
  \begin{lm}
  \label{fbounds}
  For any $\omega$ and any $n \in \ZP$,
  \[   f(n) \geq \exp \left \{ \max_{0 \leq i \leq n} \sum_{j=0}^{i-1} \log (p_j/q_j) \right\}  
  .\]
  \end{lm}
  \begin{proof}
Note that, using the non-negativity of the exponential function,
\[ f(n) = \sum_{i=0}^n \exp \left\{ \sum_{j=0}^{i-1}  \log (p_j/q_j) \right\}  
\geq   \max_{0 \leq i \leq n} \exp \left\{ \sum_{j=0}^{i-1}   \log (p_j/q_j) \right\} 
 ,\]
which yields the result by monotonicity of $x \mapsto \exp (x)$.
 \end{proof}
    
We also need 
bounds for $T(n)$. The following result
is closely related to \cite[Lemma 10]{mw2}, although in that result uniform ellipticity
was used for the upper bound; our Lemma \ref{Tbounds} 
 is essentially as sharp without the uniform ellipticity assumption, which does not hold in our main case
 of interest in the present paper.

\begin{lm}
\label{Tbounds}
For any $\omega$ and any $n \in \ZP$,
\begin{equation}
\label{Tlow}
 T(n) \geq
\exp \left\{ \max_{0 \leq i \leq n } \sum_{j=0}^{i-1} \log (p_j/q_j) \right\} .\end{equation}
On the other hand, for any $\omega$ and any $n \in \ZP$,
\begin{align}
\label{Tup1} 
T(n) & \leq  2  n^2   \exp 
\left\{  \max_{0 \leq i \leq n-1}  \sum_{j=0}^i \log (p_j/q_j) + \max_{0 \leq i \leq n-1}  \sum_{j=0}^i ( - \log (p_j/q_j))  \right\} 
\\
& \leq 2  n^2
\exp 
\left\{ 2 \max_{0 \leq i \leq n-1} \left| \sum_{j=0}^i \log (p_j/q_j) \right| \right\}
\label{Tup2} .\end{align}
\end{lm}
\begin{proof}
First we prove the upper bounds.
Since $q_{i-j}^{-1} = 1 + (p_{i-j}/q_{i-j})$, Lemma \ref{lemexp} shows  
\begin{equation}
\label{Teq} T(n) = \sum_{i=0}^{n-1} \sum_{j=0}^i \exp \left\{ \sum_{k=i-j+1}^i \log (p_k/q_k) \right\}
+  \sum_{i=0}^{n-1} \sum_{j=0}^i \exp \left\{ \sum_{k=i-j}^i \log (p_k/q_k) \right\} .\end{equation}
For the first term on the right-hand side of (\ref{Teq}) we have
\begin{align*}
 \sum_{i=0}^{n-1} \sum_{j=0}^i   \exp \left\{ \sum_{k=i-j+1}^i \log (p_k/q_k) \right\}
& \leq  \sum_{i=0}^{n-1} ( i+1 )
\max_{0 \leq j \leq i } \exp \left\{ \sum_{k=i-j+1}^i \log (p_k/q_k) \right\} \\
 & \leq  n^2 \exp \left\{ \max_{0 \leq i \leq n-1 }
\max_{0 \leq j \leq i }  \sum_{k=i-j+1}^i \log (p_k/q_k) \right\} .
\end{align*}
Here we have that
\begin{align*} \max_{0 \leq i \leq n-1 }
\max_{0 \leq j \leq i }  \sum_{k=i-j+1}^i \log (p_k/q_k)
& = \max_{0 \leq i \leq n-1 }\left( \sum_{k=0}^i \log (p_k/q_k) +  \max_{0 \leq j \leq i } \sum_{k=0}^{i-j} (-\log (p_k/q_k) ) \right) \\
& \leq  \max_{0 \leq i \leq n-1}  \sum_{k=0}^i \log (p_k/q_k) +  \max_{0 \leq i \leq n-1} \sum_{k=0}^{i} (-\log (p_k/q_k) ) .
\end{align*}
Hence
\begin{align*} &  \sum_{i=0}^{n-1} \sum_{j=0}^i   \exp \left\{ \sum_{k=i-j+1}^i \log (p_k/q_k) \right\} \\
& \quad {} \leq n^2 \exp \left\{  \max_{0 \leq i \leq n-1}  \sum_{j=0}^i \log (p_j/q_j) + \max_{0 \leq i \leq n-1}  \sum_{j=0}^i ( - \log (p_j/q_j))  \right\} .\end{align*}
A similar calculation yields the same upper bound for the second term on the right-hand side
of (\ref{Teq}), and so (\ref{Tup1}) follows. Then (\ref{Tup2}) is an immediate consequence of (\ref{Tup1}).

For the lower bound, (\ref{Teq}) shows that
\begin{align*}
T(n) \geq \max_{0 \leq i \leq n-1} \max_{0 \leq j \leq i} \exp \left\{ \sum_{k=i-j}^i \log (p_k/q_k) \right\}
 \geq \max_{0 \leq i \leq n-1}  \exp  \left\{ \sum_{k=0}^i \log (p_k/q_k)  \right\}, \end{align*} 
  which yields the result in (\ref{Tlow}).
 \end{proof}

\subsection{Asymptotics for sums of independent random variables}
\label{sec:sums}

In this section we collect some results on almost-sure asymptotics for sums of independent random variables.
For notational convenience when we come to our applications, we present the results using $\bbP$ for probability and $\bbE$
for expectation.

\subsubsection*{Upper bounds}
 
We recall the Hartman--Wintner law of the iterated logarithm (LIL): see e.g.\ \cite[p.\ 275]{kall}.

\begin{lm}
\label{lil}
Suppose that $\zeta_0, \zeta_1, \ldots$ are i.i.d.\ with $\bbE[
\zeta_0] = 0$ and $\bbE [ \zeta_0^2 ] =\sigma^2 \in (0,\infty)$. Then, $\bbP$-a.s.  
\[ \limsup_{n \to \infty} 
\frac{\sum_{i=0}^n \zeta_i}{\sigma \sqrt{ 2 n \log \log n} } = 1 . \]
\end{lm}

The most convenient statement of an analogous
result in the heavy-tailed case  is the following 
simple consequence of classical results of Feller \cite[Theorems 1 and 2]{feller}.

\begin{lm}
\label{lemfeller}
Suppose that $\zeta_0, \zeta_1, \ldots$ are i.i.d.\ random variables for which (S) holds and
 (\ref{tail1}) holds for some $\alpha \in (0,2)$.
Then for any $\eps >0$, 
$\bbP$-a.s., $| \sum_{i=0}^n \zeta_i | \leq n^{\frac{1}{\alpha}} (\log n)^{\frac{1}{\alpha} + \eps}$,
 for all but finitely many $n \in \ZP$. 
\end{lm}

\subsubsection*{Hirsch-type laws}

We will need bounds of the form $\max_{0 \leq i \leq n} \sum_{j=0}^i \zeta_j \geq a_n$, where the inequality holds
for all but finitely many $n$, almost surely. 
The following result, due to Cs\'aki \cite[Theorem 3.1]{csaki}, extends a result of Hirsch \cite{hirsch},
who imposed a 3rd moments condition. 
 
\begin{lm}
\label{lemhirsch}
Suppose that $\zeta_0, \zeta_1, \ldots$ are i.i.d.\ with $\bbE[\zeta_0] =0$ and
$\bbE [ \zeta_0^2 ] \in (0,\infty)$. For any $\eps >0$, $\bbP$-a.s., 
for all but finitely many $n \in \ZP$,
\[  \max_{0 \leq i \leq n} \sum_{j=0}^i \zeta_j \geq  n^{1/2} (\log n)^{-1-\eps} .\]
\end{lm}

In the heavy-tailed case, corresponding 
 results were obtained by Klass and Zhang \cite{kz}.
The following 
is a consequence of \cite[Theorem 5.1]{kz} (see also \cite[Example 5.2, p.\ 1872]{kz}).
\begin{lm}
\label{lemkz}
Suppose that $\zeta_0, \zeta_1, \ldots$ are i.i.d.\ for which (S) holds, and suppose that (\ref{tail1}) holds for $\alpha \in (0,2)$.
Then for any $\eps >0$, $\bbP$-a.s.,
for all but finitely many $n \in \ZP$, 
\begin{equation}
\label{hir1}
 \max_{0 \leq i \leq n} \sum_{j=0}^i \zeta_j \geq n^{\frac{1}{\alpha}} ( \log n)^{-\frac{2}{\alpha} - \eps} .\end{equation}
\end{lm}

\subsubsection*{Chung-type laws}

Finally, we also
 need bounds of the form $\max_{0 \leq i \leq n} | \sum_{j=0}^i \zeta_j | \leq a_n$, where the inequality holds
for infinitely many $n$, almost surely. 
When $\zeta_0$ has a finite second moment, the appropriate result
is the `other' law of the iterated logarithm due to Chung \cite{chung} (under a 3rd moments condition)
and Jain and Pruitt \cite{jp}:

\begin{lm} 
\label{chung}
Suppose that $\zeta_0, \zeta_1, \ldots$ are i.i.d.\ with $\bbE[ \zeta_0^2 ] = \sigma^2 \in (0,\infty)$.
Then, $\bbP$-a.s.,
\[ \liminf_{n \to \infty} \left( n^{-1/2} (\log \log n)^{1/2} \max_{0 \leq i \leq n } \left| \sum_{j=0}^i \zeta_j \right| \right) = \frac{\pi \sigma}{\sqrt{8}} . \]
\end{lm}

In the heavy-tailed case, results were obtained by
 Jain and Pruitt \cite{jp} and extended by
 Einmahl and Mason \cite{em}. The following
result is a consequence of
Corollaries 3 and 4 of \cite{em} (see also \cite[Theorem 1]{jp}).
   
 \begin{lm} 
 \label{lemem}
 Suppose that  $\zeta_0, \zeta_1, \ldots$ are i.i.d.\ satisfying (S) and (\ref{tail1}) for some $\alpha \in (0,2)$. 
 Then  there exists a 
constant $c_0 \in (0,\infty)$ for which, $\bbP$-a.s., 
\begin{equation}
\label{em1}
 \liminf_{n \to \infty} 
\left( n^{-1/\alpha} (\log \log n)^{1/\alpha}  
 \max_{0 \leq i \leq n} \left| \sum_{j=0}^i \zeta_j \right| \right) = c_0. \end{equation}
 \end{lm}

\section{Proofs of main results}
\label{sec:proofs}

\subsection{Recurrence classification}

Solomon's theorem \cite[p.\ 4]{solomon} formally gives a complete
recurrence classification for {\em any} random walk in i.i.d.\ random environment
(on $\Z$ rather than $\ZP$, but that is an inessential distinction), and
in particular shows that the process is recurrent for a.e.\ environment
if $\bbE [ \zeta_0] =0$. Solomon's result is
not so easy to use when $\bbE [ | \zeta_0 | ] = \infty$, and only applies in the case of an i.i.d.\ random environment,
 so we take a different approach.

For fixed $\omega$, our process is a nearest-neighbour random walk on $\ZP$.
The recurrence classification for such processes is classical (see e.g.\ \cite[\S I.12, pp.\ 71--76]{chungbook}) and is summarized
in the following result;
an efficient proof of part (i) of Lemma \ref{reclem}
 may be given by exploiting Lemma \ref{flem} and the Lyapunov function results of \cite[Chapter 2]{fmm}.
\begin{lm}
\label{reclem}
For fixed $\omega$, 
\begin{itemize}
\item[(i)] $X$ is recurrent if and only if
 $\lim_{n \to \infty} f(n) = \infty$;
 \item[(ii)] $X$ is positive recurrent if and only if $\sum_{n=1}^\infty  \frac{1}{D_n}  < \infty$. 
 \end{itemize}
\end{lm}

Note that the criterion  given in Lemma \ref{reclem}(ii) often appears as $\sum_{n=1}^\infty  \frac{p_0}{p_n D_n}  < \infty$, but
$\frac{1}{p_n D_n} = \frac{1}{D_n} + \frac{1}{D_{n+1}}$, and $p_i \in (0,1)$ for all $i$,
so  the criteria are equivalent.

\subsection{Almost-sure upper bounds}

To obtain  upper bounds on $X_t$, we will use the following  consequence of \cite[Lemma 9]{mw2},
whose proof (see \cite{mw2})
relies on the submartingale property of $T(X_t)$ given in Lemma \ref{1002a}.
 
\begin{lm}
\label{lemup}
Let $w : \ZP \to [0,\infty)$ be increasing, with $w(n) \to \infty$
as $n \to \infty$. Suppose that for $\bbP$-a.e.\
$\omega$,  
$T(n) \geq w(n)$ for all but finitely many $n \in \ZP$.
Then for any $\eps >0$, for $\bbP$-a.e.\ $\omega$, $\Pro$-a.s.,
for all but finitely many $t \in \ZP$,
\[ X_t \leq w^{-1} ( (2t)^{1+\eps} ) .\]
\end{lm}
\begin{proof}
Under the assumptions of the lemma,
$\bbP$-a.s.\ there exists $C_\omega \in (0,\infty)$ such that $T(n) \geq C_\omega w(n)$ for {\em all} $n \in \ZP$.
Now we can apply part (i) of \cite[Lemma 9]{mw2}.
\end{proof}

\subsection{Almost-sure lower bounds}

The next lemma
will enable us to obtain lower bounds for $X_t$ valid for infinitely
many $t$. The corresponding result in \cite{mw2} (Lemma 9 there) required
an upper bound $T(n) \leq h(n)$ valid for {\em all} $n$. 
The following result shows that a weaker bound $T(n) \leq g(n)$ (where $g(n) < h(n)$), for only {\em infinitely many} $n$,
allows us to obtain a sharper lower bound for $X_t$:
under condition (i) of Lemma \ref{lemlow}, \cite[Lemma 9]{mw2}
shows that $X_t^2 h ( X_t ) \geq t$ for infinitely many $t$, roughly
speaking a lower bound for $X_t$ of order $h^{-1} (t)$, while Lemma \ref{lemlow}
gives a bound of order $g^{-1} (t)$, which is potentially significantly larger.

\begin{lm}
\label{lemlow}
Suppose that there are non-negative, increasing functions $g$ and $h$ such that
\begin{itemize}
\item[(i)] for $\bbP$-a.e.\ $\omega$,
$T(n) \leq h(n)$ for all but finitely many $n \in \ZP$;
\item[(ii)] for $\bbP$-a.e.\ $\omega$,
$T(n) \leq g(n)$ for infinitely many $n \in \ZP$.
\end{itemize}
Suppose also that 
\begin{equation}
\label{gh} \sum_{n=1}^\infty \frac{h(n)}{g(n^{3/2})} < \infty ; \textrm{ and }
\lim_{n \to \infty} ( g(n^{3/4}) / g(n ) ) =0 .\end{equation}
Then for $\bbP$-a.e.\ $\omega$, for any $\eps>0$, $\Pro$-a.s., for infinitely many $t\in\ZP$,
\[ X_t \geq g^{-1} ( (1-\eps) t ) .\]
\end{lm}
  \begin{proof}
  First we obtain a rough upper bound on $\tau_n$ (equation
(\ref{crude}) below).
   Condition (i) and Markov's inequality yield, for $\bbP$-a.e.\ $\omega$,
\[ \Pro [ \tau_n > g (n^{3/2}) ] \leq \frac{T(n)}{g(n^{3/2})} \leq \frac{h(n)}{g(n^{3/2})} ,\]
for all $n \geq n_0$, where $n_0 := n_0 (\omega) < \infty$. Hence, for $\bbP$-a.e.\ $\omega$,
\[ \sum_{n \in \N} \Pro [ \tau_n > g (n^{3/2}) ] \leq n_0 + \sum_{n \in \N} \frac{h(n)}{g(n^{3/2})} < \infty ,\]
by the first condition in (\ref{gh}). Thus by the Borel--Cantelli lemma,
for $\bbP$-a.e.\ $\omega$,
\begin{equation}
\label{crude}
\tau_n \leq g (n^{3/2} ),
\end{equation}
for all $n \geq N_0$ where $N_0 := N_0 (\omega)$ has $\Pro [ N_0 (\omega ) < \infty] =1$.

Now we use (\ref{crude}) and condition (ii) to show that $\tau_n$ is in fact much smaller
than the bound in (\ref{crude}) for {\em infinitely many} $n$.
  Let $\eps>0$.
  Condition (ii) implies that for $\bbP$-a.e.\ $\omega$ there exist $n_i := n_i (\omega)$, $i \in \N$,
  such that $n_{i+1} > n_i^2$ for all $i$ and
  $T(n_i) \leq g(n_i)$ for all $i$ (the $n_i$ are a subsequence of that specified by (ii)
  chosen so as to have very large spacings).
  By Markov's inequality,
\[ \Pro [ \tau_{n_{i+1}} -\tau_{n_i} > (1+\eps) g ( n_{i+1}) ]
\leq \Pro [ \tau_{n_{i+1}}  > (1+\eps) g ( n_{i+1}) ] \leq 1/(1+\eps) .\]
It follows that, for all $i \in \N$, 
\begin{equation}
\label{eq3}
\Pro [ \tau_{n_{i+1}} -\tau_{n_i} \leq (1+\eps) g ( n_{i+1}) ]
 \geq \eps', \end{equation}
where $\eps' >0$ depends only on $\eps$. So by (\ref{eq3}), for $\bbP$-a.e.\ $\omega$,
 for any $\eps>0$,  
\begin{equation}
\label{eq4} \sum_{i \in \N} \Pro [ \tau_{n_{i+1}} -\tau_{n_i} \leq (1+\eps) g ( n_{i+1}) ] = \infty. \end{equation}

Under $\Pro$, the random variables $\tau_{n_{i+1}} -\tau_{n_i}$, $i \in \N$,
are independent, by the strong Markov property. Hence
(\ref{eq4}) and the Borel--Cantelli lemma imply that, for $\bbP$-a.e.\ $\omega$, $\Pro$-a.s.,
for infinitely many $i$,
\[ \tau_{n_{i+1}} -\tau_{n_i} \leq (1+\eps) g ( n_{i+1}) . \]
 Together with (\ref{crude}) this implies that, for $\bbP$-a.e.\ $\omega$, 
$\Pro$-a.s., for infinitely many $i$,
\[ \tau_{n_{i+1}} \leq (1+\eps) g ( n_{i+1}) + g (n_i^{3/2})
\leq (1+\eps) g ( n_{i+1}) + g (n_{i+1}^{3/4}) ,\]
since $n_i < n_{i+1}^{1/2}$ and $g$ is increasing. Hence by the second condition
in (\ref{gh}), we finally obtain that, for any $\eps>0$, for $\bbP$-a.e.\ $\omega$, $\Pro$-a.s.,
for infinitely many $n$,
\[ \tau_{n} \leq (1 + 2 \eps) g (n) .\]
Since $\tau_n \neq \tau_m$ for any $n \neq m$, we conclude that
for any $\eps>0$, for $\bbP$-a.e.\ $\omega$, $\Pro$-a.s.,
\[ t = \tau_{X_t} \leq (1+ 2 \eps) g (X_t ) ,\]
for infinitely many $t$, and the result follows.   
  \end{proof}

\subsection{Proofs of theorems on recurrence and transience}

Recall the notation from Section \ref{sec:model}.
First we appeal to a result of Solomon \cite{solomon} to prove Theorem \ref{thm2}.

\begin{proof}[Proof of Theorem \ref{thm2}.]
By assumption (S),  $\sum_{i=0}^n \zeta_i$ has the same
distribution as $-\sum_{i=0}^n \zeta_i$, so in particular
\[ \sum_{n = 0} ^{ \infty} n^{-1} \bbP \left[ \sum_{i=0}^n \zeta_i > 0 \right]  = \sum_{n = 0} ^{ \infty} n^{-1} \bbP \left[ \sum_{i=0}^n \zeta_i < 0 \right] .\]
Recurrence follows from Solomon's theorem \cite{solomon}. In fact, Solomon's theorem is stated for the RWRE on $\Z$, but 
the result carries across to the present setting. 
\end{proof}

The remainder of our results on recurrence and transience will use Lemma \ref{reclem}
and an analysis of the quantity $f(n)$. Recall that
\begin{equation}
\label{fn}
f(n) = \sum_{i=0}^n \exp \left\{ \sum_{j=0}^{i-1} \log (p_j/q_j ) \right\}.
\end{equation}

\begin{proof}[Proof of Theorem \ref{thm3}.]
Here $\phi(n) \equiv 0$, so $\log (p_n/ q_n) = \zeta_n$, $\bbP$-a.s.
It follows from the conditions of the theorem that for some $\eps>0$,
$\bbP [ \zeta_0 < -x ] \geq c x^{-(1 \wedge \alpha_+ -\eps)}$ for some $c>0$
and all $x$ large enough, while $\bbE [ (\zeta_0^+ )^p ] < \infty$ for any $p < \alpha_+$.
Then a result of Derman and Robbins (see \cite[Theorem 3.2.6, p.\ 133]{stout})
implies that, for $\bbP$-a.e.\ $\omega$, 
$\sum_{i=0}^n \zeta_i < - n$,
for all $n \geq n_0$ where $n_0 := n_0 (\omega) < \infty$.
Hence, for $\bbP$-a.e.\ $\omega$, by (\ref{fn}),
\[ f(n) \leq \sum_{i=0}^{n_0}  \exp \left\{ \sum_{j=0}^{i-1} \zeta_j \right\}
+ \sum_{i \geq n_0+1} \exp \{ - i \} < \infty .\]
   Transience follows from Lemma \ref{reclem}.
\end{proof}

When $\phi(n) > 0$, the analysis of (\ref{fn}) becomes a little more involved.
Central to our arguments for most of our theorems are almost-sure estimates of the following kind.

\begin{lm}
\label{sums}
Suppose that (S) holds. Suppose that $\phi (n) = n^{-\beta}$ for $\beta \in (0,1)$.
\begin{itemize}
\item[(i)] Suppose that either (a) $\bbE[ \zeta_0^2] <\infty$ and $\beta 
< \frac{1}{2}$; or (b)
(\ref{tail1}) holds with $\alpha \in (1,2)$ and $\beta < 1 -\frac{1}{\alpha}$. Then
for $c_\beta := \frac{1}{1-\beta} \in (1, 2)$, $\bbP$-a.s., as $n \to \infty$,
\begin{equation}
\label{eq21}
  \sum_{j=0}^n \log (p_j/q_j) = (\rho c_\beta + o(1)) n^{1-\beta} .\end{equation}
\item[(ii)] Suppose that $\bbE[ \zeta_0^2] < \infty$ and $\beta > \frac{1}{2}$. Then, $\bbP$-a.s.,
as $n \to \infty$,
\begin{equation}
\label{eq22}
 \sum_{j=0}^n \log (p_j/q_j) =  \sum_{j=0}^n \zeta_j + O (n^{1-\beta}) .\end{equation}
\item[(iii)] Suppose that (\ref{tail1}) holds with $\alpha \in (1,2)$ and $\beta > 1 -\frac{1}{\alpha}$. 
Then there exists $\eps>0$ such that, $\bbP$-a.s., as $n \to \infty$,
\begin{equation}
\label{eq233}
 \sum_{j=0}^n \log (p_j /q_j ) = \sum_{j=0}^n \zeta_j + O ( n^{(1/\alpha)-\eps}) .\end{equation}
\end{itemize}
\end{lm}
\begin{proof}
For $k \in \{0,1\}$, set
\begin{equation}
\label{Nkdef}
 N_k(n) := \# \{ j \in \{0,1,\ldots, n\} :  \chi_j = k \} .\end{equation}
 Then $N_1 (n) = \sum_{j=0}^n \chi_j$ is a sum of independent,
$\{0,1\}$-valued random variables and $\bbE [ N_1(n) ]  = \sum_{j=0}^n \phi (j)$. 
  We have that
\begin{equation}
\label{eq43}
 \sum_{j=0}^n \log (p_j/q_j) = \rho N_1(n) + \sum_{j=0}^n \zeta_j (1- \chi_j) .\end{equation}
Here $\bbE [ N_1 (n) ] \sim c_\beta n^{1-\beta}$ (where $c_\beta = (1-\beta)^{-1}$ as in the statement of the lemma)
and $\Var[ \chi_n ] = (1-\phi(n) )\phi (n) \sim n^{-\beta}$.
The Kolmogorov convergence criterion for
sums of independent
random variables  with finite second moments (see e.g.\ \cite[Corollary 4.22, p.\ 73]{kall})
implies that, for any $\eps >0$,
 $\bbP$-a.s., $| N_1 (n) - \bbE [ N_1 (n) ] | = O ( n^{\frac{1-\beta}{2} + \eps } )$. 
Choosing $\eps < (1-\beta)/2$, we thus see that
$N_1 (n) = (c_\beta  +o(1)) n^{1-\beta}$, $\bbP$-a.s.
Then by (\ref{eq43}),
\begin{equation}
\label{eq43b}
 \sum_{j=0}^n \log (p_j/q_j) =  (\rho c_\beta  +o(1)) n^{1-\beta} + \sum_{j=0}^n \zeta_j (1- \chi_j) .\end{equation}
We need to deal with the final sum on the right-hand side of (\ref{eq43b}).

Consider first the case where $\bbE [ \zeta_0^2] < \infty$.
Then
$\bbE [ \zeta_n (1-\chi_n) ] = 0$ and $\Var [ \zeta_n (1-\chi_n) ] = \bbE[ \zeta_0^2 ] \bbE [ (1-\chi_n)^2 ] = \bbE[ \zeta_0^2 ] 
+ O (n^{-\beta})$,
so another application of the  Kolmogorov convergence criterion   implies that, for any $\eps >0$, 
$\bbP$-a.s., $| \sum_{j=0}^n \zeta_j (1- \chi_j) | = O ( n^{(1/2) +\eps} )$,
which is
$o ( n^{1-\beta} )$, for a suitably small choice of $\eps$, provided $\beta \in (0,1/2)$. 
So in this case from  (\ref{eq43b}) we see that, $\bbP$-a.s., (\ref{eq21})
holds.
This proves part (i) of the lemma in case (a).
In case (b), note that  $( \sum_{j=0}^n \zeta_j (1-\chi_j) )_{n \in \N}$ has the same distribution
as $( \sum_{j=0}^{N_0(n)} \zeta_j )_{n \in \N}$. 
This fact together with
 Lemma \ref{lemfeller} 
implies that, for any $\eps>0$, $\bbP$-a.s., 
\[ \left| \sum_{j=0}^n \zeta_j (1-\chi_j) \right| = O( N_0(n) ^{(1/\alpha) + \eps } ) = O( n^{(1/\alpha) + \eps } ) .\]
With (\ref{eq43b}) we see that (\ref{eq21}) again holds, since if
 $1- \beta > 1/\alpha$, we can choose $\eps$ small enough so that $n^{(1/\alpha) + \eps } = o (n^{1-\beta} )$.
Thus we verify part (i) of the lemma in case (b) also.

Next we prove part (ii). We obtain from (\ref{eq43b}), $\bbP$-a.s., 
\begin{equation}
\label{eq43c}
 \sum_{j=0}^n \log (p_j/q_j) =  O ( n^{1-\beta}) + \sum_{j=0}^n \zeta_j - \sum_{j=0}^n \zeta_j \chi_j .\end{equation}
Suppose once more that $\bbE [\zeta_0^2] <\infty$. 
To deal with the final sum in (\ref{eq43c}), 
note that $\bbE[ \zeta_n \chi_n] = 0$ and $\Var [ \zeta_n \chi_n ] = \bbE [ \zeta_0^2] \bbE [ \chi^2_n]
= O (n^{-\beta} )$.
Kolmogorov's convergence criterion now implies that, for any $\eps>0$, $\bbP$-a.s., $| \sum_{j=0}^n \zeta_j  \chi_j | = O ( n^{ \frac{1-\beta}{2} +\eps} )$.
Combining this with (\ref{eq43c}) we obtain (\ref{eq22}), which proves part (ii) of the lemma.

Finally, we prove part (iii).
Again consider (\ref{eq43c}).
Here, since $( \sum_{j=0}^n \zeta_j \chi_j )_{n \in \N}$ has the same distribution
as $( \sum_{j=0}^{N_1(n)} \zeta_j )_{n \in \N}$,
Lemma \ref{lemfeller} 
implies that, for any $\eps>0$, $\bbP$-a.s., 
\[ \left| \sum_{j=0}^n \zeta_j \chi_j \right| = O( N_1(n) ^{(1/\alpha) + \eps } ) = O( n^{(1-\beta)((1/\alpha) +\eps)} ) ,\]
and since $\beta >0$, $\alpha < \infty$ we may choose $\eps$ sufficiently small (less than $\frac{\beta}{2\alpha}$, say)
 that
this last term is $O (n^{(1/\alpha)-\eps})$.
Now part (iii) follows from (\ref{eq43c}), since $\beta > 1- \frac{1}{\alpha}$.
\end{proof}

Now we can prove the rest of our main theorems on recurrence and transience.

\begin{proof}[Proof of Theorem \ref{pert0}.]
For part (i), first consider the case $\beta < \frac{1}{2}$. In this case, Lemma \ref{sums}(i)(a) applies and so (\ref{eq21}) holds.
It follows from (\ref{fn}) that if $\rho<0$, then $\max_{n \in \ZP} f(n) < \infty$ for $\bbP$-a.e.\ $\omega$.
On the other hand, if $\rho >0$, then 
\begin{equation}
\label{flow1}
f(n) = \exp \{ (\rho c_\beta +o(1) ) n^{1-\beta} \}, \end{equation}
 for $\bbP$-a.e.\ $\omega$.
Hence for $\bbP$-a.e.\ $\omega$,  $f(n) \to \infty$ as $n \to \infty$.
Part (i) now follows from  Lemma \ref{reclem} and the fact that $\delta$ and $\rho$ have opposite signs.

For part (ii), suppose that $\beta > \frac{1}{2}$.  
Now, by Lemma \ref{fbounds} and (\ref{eq22}), $\bbP$-a.s.,
\[ f(n) 
\geq \exp \left\{ \max_{0 \leq i \leq n } \sum_{j=0}^{i-1} \zeta_j - O (n^{1-\beta})\right\} .\]
Now  Lemma \ref{lemhirsch}, with the fact that $\beta > \frac{1}{2}$, 
shows that for any $\rho$, for any $\eps>0$,
for $\bbP$-a.e.\ $\omega$,
for all but finitely many $n$, 
\begin{equation}
\label{flow2}
f(n) \geq \exp \{ n^{1/2} (\log n)^{-1-\eps} \}. \end{equation}
Since the right-hand side of (\ref{flow2}) tends to infinity with $n$,
another application of Lemma \ref{reclem} yields recurrence, and completes the proof of part (ii).
\end{proof}

\begin{proof}[Proof of Theorem \ref{pert1}.]
Suppose now that (\ref{tail1}) holds for $\alpha \in (1,2)$. 
Part (i) follows in the same way as the proof of part (i) of Theorem \ref{pert0},
since, by Lemma \ref{sums}(i)(b),
 (\ref{eq21}) also holds in the case where  $\beta < 1 - \frac{1}{\alpha}$.

It remains to prove part (ii). Now 
Lemma \ref{fbounds} with Lemma \ref{lemkz} and (\ref{eq233}) 
implies that, for some $\eps >0$ and any $\eps'>0$, 
$\bbP$-a.s., for infinitely many $n \in \ZP$,
\[ f(n) \geq \exp \left\{  \max_{0 \leq i \leq n}  \sum_{j=0}^{i-1} \zeta_j - O ( n^{(1/\alpha)-\eps} )
\right\}
 \geq \exp \left\{ n^{1/\alpha} (\log n)^{-(2/\alpha) -\eps'} \right\} .\]
So, for  any $\rho$, $f(n) \to \infty$ $\bbP$-a.s. 
Recurrence follows from Lemma \ref{reclem}.
\end{proof}

\begin{proof}[Proof of Theorem \ref{thm4}.]
Now $\phi (n) = 1 -n^{-\beta}$, and (\ref{tail1}) holds with $\alpha \in (0,1)$. Recall the
definition of $N_k(n)$ from (\ref{Nkdef}) and the representation (\ref{eq43}). In the present setting,
the sequence $N_1(n)$ is distributed as the sequence $N_0(n)$ in the setting of Lemma \ref{sums}
(where $\phi(n) = n^{-\beta}$), and so the argument in the proof of Lemma \ref{sums} shows that,
$\bbP$-a.s.,
$N_0 (n) \sim c_\beta n^{1-\beta} $,
$N_1(n) \sim n$,  and
\[ \sum_{j=0}^n \log (p_j/q_j) = ( \rho +o(1) ) n + \sum_{j=0}^n \zeta_j (1 - \chi_j ) .\]
Here $(\sum_{j=0}^n \zeta_j (1 - \chi_j ))_{n \in \N}$
has the same distribution as $( \sum_{j=0}^{N_0(n)} \zeta_j )_{n\in \N}$.

First we prove part (ii) of the theorem. 
Lemma \ref{lemfeller} shows that, for any $\eps>0$, $\bbP$-a.s., 
\[ \left| \sum_{j=0}^n \zeta_j (1 - \chi_j ) \right| = O(  N_0 (n) ^{(1/\alpha) + \eps} )
= O \left( n^{\frac{1-\beta}{\alpha} +\eps } \right),\]
which is $o(n)$ for $\beta > 1-\alpha$ and $\eps$ small enough.
Then, by (\ref{fn}), we have that, for $\bbP$-a.e.\ $\omega$,
 $f(n) \to \infty$ if $\rho >0$ but $\max_{n \in \ZP} f(n) < \infty$ if $\rho <0$.
Lemma \ref{reclem} then gives the proof of part (ii) of the theorem.

On the other hand, Lemma \ref{lemkz} shows that,
$\bbP$-a.s., for all but finitely many $n$,
\[ \max_{0 \leq i \leq n } \sum_{j=0}^{i-1} \zeta_j (1-\chi_j)
 \geq ( N_0 (n) )^\frac{1}{\alpha} ( \log N_0 (n) )^{-\frac{3}{\alpha}} \geq n^{\frac{1-\beta}{\alpha}} (\log n)^{-\frac{4}{\alpha} } ,\]
since $N_0 (n) \sim c_\beta n^{1-\beta} $. Then by Lemma \ref{fbounds}, since $1-\beta > \alpha$,
\[ f(n) \geq \exp \left\{ \max_{0 \leq i \leq n} \sum_{j=0}^{i-1} \zeta_j  - O (n) \right\}
\geq \exp \left\{ n^{\frac{1-\beta}{\alpha}} (\log n)^{-\frac{5}{\alpha} } \right\} ,\]
for all $n$ sufficiently large. So $f(n) \to \infty$ for $\bbP$-a.e.\ $\omega$, for any $\rho$,
and Lemma \ref{reclem} shows recurrence, completing the proof of part (i) of the theorem.
\end{proof}

\subsection{Proofs of theorems on almost-sure bounds}

We obtain our almost-sure bounds on $X_t$ via 
Lemmas \ref{lemup} and \ref{lemlow}. 
This requires an analysis of the expected hitting times $T(n)$.
The next two lemmas collect bounds on $T(n)$ that we will use. 

\begin{lm}
\label{lem7a}
Suppose that (S) holds, and that $\bbE [   \zeta_0 ^2 ] \in (0,\infty)$.
Suppose that $\phi(n) = n^{-\beta}$ for $\beta \in (0,1)$. 
\begin{itemize}
\item[(i)] Suppose that $\beta  < \frac{1}{2}$
and  $\rho>0$. Then with $c_\beta = \frac{1}{1-\beta}$, $\bbP$-a.s., as $n \to \infty$,
\[  T(n) = \exp \{ (\rho c_\beta +o(1)) n^{1-\beta} \}. \]
\item[(ii)] Suppose that $\beta > \frac{1}{2}$. Then (for any $\rho$), for any $\eps>0$,
$\bbP$-a.s., for all but finitely many $n \in \ZP$, 
\[ \exp \{ n^{1/2} (\log n)^{-1 -\eps} \} \leq T(n) \leq 
\exp \{ n^{1/2} (\log \log n)^{(1/2) +\eps} \}. \]
Moreover, there is some $c \in (0,\infty)$ for which, $\bbP$-a.s., for infinitely many $n \in \ZP$,
\[ T(n) \leq \exp \{ c n^{1/2 } (\log \log n)^{-1/2}  \}.\]
\end{itemize}
\end{lm}
\begin{proof}
The lower bounds for $T(n)$ follow in the same way as the lower bounds for $f(n)$ obtained in (\ref{flow1}) and (\ref{flow2}),
using the lower bound for $T(n)$   in (\ref{Tlow}) in place of the identical 
  lower bound for $f(n)$   in
 Lemma \ref{fbounds}. 

Next we prove the 
upper bounds on $T(n)$.
Under the conditions of part (i) of the lemma, Lemma \ref{sums}(i)(a) shows that (\ref{eq21}) holds. Then from the upper bound for $T(n)$ in (\ref{Tup1})
 we obtain, since $\rho >0$, 
\[ T(n) \leq 2n^2 \exp \left\{ ( \rho c_\beta + o(1) ) n^{1-\beta} + O(1) \right\} = \exp \left\{ ( \rho c_\beta + o(1) ) n^{1-\beta}   \right\} ,\]
matching the lower bound and completing the proof of part (i). 
Under the conditions of part (ii) of the lemma, Lemma \ref{sums}(ii) shows  that (\ref{eq22}) holds. Thus from the upper bound in (\ref{Tup2})
we have that $\bbP$-a.s., for all but finitely many $n \in \ZP$,
\begin{equation}
\label{eq23}
 T(n) \leq \exp \left\{ 2 \max_{0 \leq i \leq n} \left| \sum_{j=0}^i \zeta_j \right| + O( n^{1-\beta}) \right\} .\end{equation}
Here the LIL (Lemma \ref{lil})   shows that, $\bbP$-a.s., $ \max_{0 \leq i \leq n} | \sum_{j=0}^i \zeta_j | = O (n^{1/2} (\log \log n)^{1/2} )$,
which with (\ref{eq23}) yields the `all but finitely often' upper bound in part (ii).

On the other hand, the Chung--Jain--Pruitt law (Lemma \ref{chung}) shows that for some $c \in (0,\infty)$, $\bbP$-a.s., 
for infinitely many $n \in \ZP$,
 $\max_{0 \leq i \leq n} | \sum_{j=0}^i \zeta_j | \leq c n^{1/2} ( \log \log n)^{-1/2}$, 
which with (\ref{eq23}) yields the final
 statement in part (ii).
\end{proof}

\begin{lm}
\label{lem7b}
Suppose that (S) holds, and that (\ref{tail1}) holds for $\alpha \in (1,2)$.
Suppose that $\phi(n) = n^{-\beta}$ for $\beta \in (0,1)$. 
\begin{itemize}
\item[(i)] Suppose that $\beta < 1-\frac{1}{\alpha}$
and  $\rho>0$. Then with $c_\beta = \frac{1}{1-\beta}$, $\bbP$-a.s., as $n \to \infty$,
\[   T(n) = \exp \{ (\rho c_\beta +o(1)) n^{1-\beta} \}. \]
\item[(ii)] Suppose that $\beta > 1-\frac{1}{\alpha} $. Then (for any $\rho$), for any $\eps>0$,
$\bbP$-a.s., for all but finitely many $n \in \ZP$, 
\[ \exp \{ n^{1/\alpha} (\log n)^{-(2/\alpha)-\eps} \} \leq T(n) \leq 
\exp \{ n^{1/\alpha} (\log n)^{(1/\alpha) +\eps} \}. \]
Moreover, there is some $c \in (0,\infty)$ for which, $\bbP$-a.s., for infinitely many $n \in \ZP$,
\[ T(n) \leq \exp \{ c n^{1/\alpha} (\log \log n)^{-1/\alpha} \}.\]
\end{itemize}
\end{lm}
\begin{proof}
The proof is analogous to that of Lemma \ref{lem7a}, so we only point out the main differences, which appear in the upper bounds in part (ii).

Under the conditions of part (ii) of the lemma, Lemma \ref{sums}(iii) shows that (\ref{eq233}) holds. Then the upper bound in (\ref{Tup2}) implies
that, for some $\eps>0$,
$\bbP$-a.s., for all but finitely many $n \in \ZP$,
\begin{equation}
\label{eq55}
T(n) \leq \exp \left\{ 2 \max_{0 \leq i \leq n} \left| \sum_{j=0}^i \zeta_j \right| + O (  n^{(1/\alpha)-\eps} ) \right\} .\end{equation}
  Lemma \ref{lemfeller} shows that for any $\eps'>0$, $\bbP$-a.s., 
$\max_{0 \leq i \leq n} | \sum_{j=0}^i \zeta_j | = O (n^{1/\alpha} (\log n)^{(1/\alpha)+\eps' } )$,
which with (\ref{eq55})  yields the `all but finitely often' upper bound in part (ii).

On the other hand, Lemma \ref{lemem} implies that 
  for some $c \in (0,\infty)$, $\bbP$-a.s., for infinitely many $n \in \ZP$,
 $\max_{0 \leq i \leq n} | \sum_{j=0}^i \zeta_j | \leq c n^{1/\alpha} ( \log \log n)^{-1/\alpha}$, which with (\ref{eq55}) yields the final
 statement in part (ii).
\end{proof}

Now we can complete the proofs of Theorems \ref{pert2a} and \ref{pert2b}. We give full details first for the case
of Theorem \ref{pert2b}.

\begin{proof}[Proof of Theorem \ref{pert2b}.]
We
use Lemmas \ref{lemup} and \ref{lemlow} with the bounds on $T(n)$ given in Lemma \ref{lem7b}.
First we prove part (ii) of the theorem.
 Lemma \ref{lem7b}(ii) shows that 
the hypothesis of Lemma \ref{lemup} holds with $w(n) = \exp \{ n^{1/\alpha} (\log n)^{-(2/\alpha) - \eps} \}$
for any fixed $\eps>0$. Now we have that
$w ( (\log n)^{\alpha}  (\log \log n)^{2+2\alpha \eps} ) > n$
for all $n$ large enough, so that
\[ w^{-1} (n) \leq (\log n)^{\alpha} (\log \log n)^{2+2\alpha \eps} ,\]
for all $n$ large enough.
So we may apply Lemma \ref{lemup}, yielding the upper bound in part (ii) of the theorem.
For the lower bound, we have from Lemma \ref{lem7b}(ii) 
that the hypotheses of Lemma \ref{lemlow}(i) and (ii) hold with 
$h(n) = \exp \{ n^{1/\alpha} (\log n)^{(1/\alpha) + \eps} \}$ and
$g(n) = \exp \{ c n^{1/\alpha} (\log \log n)^{-1/\alpha } \}$, respectively.
It is also easy to check that, for these choices of $g$ and $h$,
 $h(n)/g(n^{3/2})$ is summable and $g(n^{3/4})/g(n) \to 0$.
 Hence Lemma \ref{lemlow} applies; we thus need to estimate $g^{-1} (t)$.
Now, for suitable $A \in (0,\infty)$,
$g ( A (\log n)^{\alpha} \log \log \log n ) < n$
for all $n$  large enough, so that
\[ g^{-1} (t) \geq A (\log t)^{\alpha} \log \log \log t  ,\]
for all $t$ large enough. The lower bound in part (ii) of the theorem follows.

The argument for part (i) of the theorem is similar. By Lemma \ref{lem7b}(i), for any $\eps>0$,
we can take $w (n) = \exp \{ (\rho c_\beta - \eps) n^{1-\beta} \}$
and $g(n) = h(n) = \exp \{ (\rho c_\beta + \eps) n^{1-\beta} \}$ for which the hypotheses
of Lemmas \ref{lemup} and \ref{lemlow} are satisfied. Lemma \ref{lemup} then shows that, for any $\eps>0$,
for $\bbP$-a.e.\ $\omega$, $\Pro$-a.s.,
\[ X_t \leq (1+ \eps) ( \rho c_\beta )^{-\frac{1}{1-\beta}} ( \log t)^{\frac{1}{1-\beta}} ,\]
for all but finitely many $t \in \ZP$. On the other hand, Lemma \ref{lemlow} shows that, for any $\eps>0$,
for $\bbP$-a.e.\ $\omega$, $\Pro$-a.s.,
\[ X_t \geq (1- \eps) ( \rho c_\beta )^{-\frac{1}{1-\beta}} ( \log t)^{\frac{1}{1-\beta}} ,\]
for infinitely many $t \in \ZP$. Combining these two bounds, since $\eps>0$ was arbitrary, gives the claimed $\limsup$ result.
\end{proof}

\begin{proof}[Proof of Theorem \ref{pert2a}.]
This is similar to the previous proof, but using Lemma \ref{lem7a} in place of Lemma \ref{lem7b}.
\end{proof}

\section*{Acknowledgements}

We benefited in the early stages of this work from interesting 
discussions with
Iain  MacPhee, who sadly passed away on 13th January 2012; we dedicate this paper to Iain, in memory of
our valued  colleague and in gratitude for his generosity.

\end{document}